\documentclass{amsart}
\usepackage{amsmath}
\usepackage{amsfonts}

\newtheorem{theorem}{Theorem}
\theoremstyle{plain}

\newtheorem{definition}{Definition}

\newtheorem{remark}{Remark}

\numberwithin{equation}{section}

\begin{document}
\title[On Jensen Functional, Convexity and Uniform Convexity]{On Jensen
Functional, Convexity and Uniform Convexity}
\author{Shoshana Abramovich}
\address{Department of Mathematics, University of Haifa, Haifa, Israel}
\email{abramos@math.haifa.ac.il }
\date{May 22, 2024}

\begin{abstract}
In this paper we improve\ results related to Normalized Jensen Functional
for convex functions and Uniformly Convex Functions.
\end{abstract}

\maketitle

\section{\protect\bigskip \textbf{Introduction}}

In this paper we extend and refine Jensen Functional type inequalities that
appeared in \cite{AD}, \cite{D} and \cite{BMP}. \textbf{Jensen functional }%
is:%
\begin{equation*}
J_{n}\left( f,\mathbf{x},\mathbf{p}\right) =\sum_{i=1}^{n}p_{i}f\left(
x_{i}\right) -f\left( \sum_{i=1}^{n}p_{i}x_{i}\right) .
\end{equation*}

We start with some theorems, definitions and notations that appeared in
these papers and are used here.

\begin{theorem}
\label{Th1} \cite{D} \textit{Consider the normalized Jensen functional where 
}$f:C\longrightarrow 
\mathbb{R}
$\textit{\ is a convex function on the convex set }$C$ in a real linear
space,\textit{\ }$\mathbf{x}=\left( x_{1},...,x_{n}\right) \in C^{n},$ and%
\textit{\ \ }$\mathbf{p}=\left( p_{1},...,p_{n}\right) ,$\textit{\ \ }$%
\mathbf{q}=\left( q_{1},...,q_{n}\right) $\textit{\ are non-negative
n-tuples satisfying }$\sum_{i=1}^{n}p_{i}=1,$\textit{\ \ }$%
\sum_{i=1}^{n}q_{i}=1,$\textit{\ \ }$q_{i}>0,$\textit{\ \ }$i=1,...,n$%
\textit{. Then } 
\begin{equation}
MJ_{n}\left( f,\mathbf{x},\mathbf{q}\right) \geq J_{n}\left( f,\mathbf{x},%
\mathbf{p}\right) \geq mJ_{n}\left( f,\mathbf{x},\mathbf{q}\right) ,\qquad
\label{1.1}
\end{equation}%
provided that 
\begin{equation*}
m=\min_{1\leq i\leq n}\left( \frac{p_{i}}{q_{i}}\right) ,\quad M=\max_{1\leq
i\leq n}\left( \frac{p_{i}}{q_{i}}\right) .
\end{equation*}
\end{theorem}

In, Section 2 we use Theorem \ref{Th2} and the following notations:

Let $\mathbf{x}_{\uparrow }=\left( x_{\left( 1\right) },...,x_{\left(
n\right) }\right) $ be the \textit{increasing rearrangement} of $\mathbf{x=}%
\left( x_{1},...,x_{n}\right) .$ $\ $Let\ $\pi $ be the permutation that
transfers $\mathbf{x}$ into $\mathbf{x}_{\uparrow }$\ and let $\left( 
\overline{p}_{1},...,\overline{p}_{n}\right) $ and $\left( \overline{q}%
_{1},...,\overline{q}_{n}\right) $ be the $n$-tuples obtained by the same
permutation \ $\pi $\ \ on $\left( p_{1},...,p_{n}\right) $ and $\left(
q_{1},...,q_{n}\right) $ respectively. Then for an $n$-tuple \ $\mathbf{x}%
=\left( x_{1},...,x_{n}\right) ,$\ \ $x_{i}\in I$\ ,\ \ $i=1,...,n$ \ where $%
I$ \ is an interval in $%
\mathbb{R}
$ \ we get the following results:

\begin{theorem}
\label{Th2} \cite[Theorem 4]{AD} Let $\mathbf{p}=\left(
p_{1},...,p_{n}\right) ,$ where $\ 0\leq \sum_{j=1}^{i}\overline{p}_{j}\leq
1,$\ \ $i=1,...,n,$\ \ $\sum_{i=1}^{n}p_{i}=1,$ and\ $\mathbf{q}=\left(
q_{1},...,q_{n}\right) ,$ \ $0<\sum_{j=1}^{i}\overline{q}_{j}<1,$\ \ $%
i=1,...,n-1$,\ \ $\sum_{i=1}^{n}q_{i}=1,$ and $\ \mathbf{p}\neq $\ \ $%
\mathbf{q}.$ \ Denote 
\begin{equation*}
m_{i}:=\frac{\sum_{j=1}^{i}\overline{p}_{j}}{\sum_{j=1}^{i}\overline{q}_{j}}%
,\qquad \overline{m}_{i}=:\frac{\sum_{j=i}^{n}\overline{p}_{j}}{%
\sum_{j=i}^{n}\overline{q}_{j}},\qquad i=1,...,n
\end{equation*}%
where $\left( \overline{p}_{1},...,\overline{p}_{n}\right) $ and $\left( 
\overline{q}_{1},...,\overline{q}_{n}\right) $ are as denoted above, and 
\begin{equation*}
m^{\ast }:=\min_{1\leq i\leq n}\left\{ m_{i},\overline{m_{i}}\right\}
,\qquad M^{\ast }:=\max_{1\leq i\leq n}\left\{ m_{i},\overline{m_{i}}%
\right\} .
\end{equation*}%
If \ $\mathbf{x}=\left( x_{1},...,x_{n}\right) $\ \ is any n-tuple in $%
I^{n}, $\ \ where $I$ is an interval in $%
\mathbb{R}
,$ then 
\begin{equation}
M^{\ast }J_{n}\left( f,\mathbf{x},\mathbf{q}\right) \geq J_{n}\left( f,%
\mathbf{x},\mathbf{p}\right) \geq m^{\ast }J_{n}\left( f,\mathbf{x},\mathbf{q%
}\right) ,  \label{1.2}
\end{equation}%
where\ \ $f:I\longrightarrow 
\mathbb{R}
$\ \ is a convex function on the interval $I$.
\end{theorem}

\begin{remark}
\label{Rem1} \cite[Remark 3]{AD} If \ $\underset{1\leq i\leq n}{min}\left( 
\frac{p_{i}}{q_{i}}\right) =\frac{\overline{p}_{k}}{\overline{q}_{k}},$\ $%
k\neq 1,n$\ or\ $\underset{1\leq i\leq n}{\max }\left( \frac{p_{i}}{q_{i}}%
\right) =\frac{\overline{p}_{s}}{\overline{q}_{s}},$ $\ s\neq 1,n$ \ then it
is clear that for $p_{i}\geq 0,$ \ and \ $q_{i}>0,$ we get that$\ \ m^{\ast
}>m$ \ and $M^{\ast }<M$ \ and in these cases (\ref{1.2}) refines (\ref{1.1}%
).
\end{remark}

In Section 3 we deal with refinements of Theorem \ref{Th1} through Uniform
Convexity.

\begin{definition}
\label{Def1} \cite{A} Let $I=\left[ a,b\right] \subset 
\mathbb{R}
$ be an interval and $\Phi :\left[ 0,b-a\right] \rightarrow 
\mathbb{R}
$ be a function. A function $f:\left[ a,b\right] \rightarrow 
\mathbb{R}
$ is said\ to be \textbf{generalized }$\Phi $\textbf{-uniformly convex}\ if:%
\begin{eqnarray*}
tf\left( x\right) +\left( 1-t\right) f\left( y\right) &\geq &f\left(
tx+\left( 1-t\right) y\right) +t\left( 1-t\right) \Phi \left( \left\vert
x-y\right\vert \right) \\
\text{for \ }x,y &\in &I\text{ \ and }t\in \left[ 0,1\right] \text{.}
\end{eqnarray*}%
If in addition $\Phi \geq 0$, then $f$\ is said to be $\Phi $\textbf{%
-uniformly convex}\textit{.}
\end{definition}

\begin{remark}
\label{Rem2} It is proved in \cite{Z} that when $f$ is $\Phi $-uniformly
convex,\textbf{\ }there is always a modulus $\Phi $ which is increasing and $%
\Phi \left( 0\right) =0$. We use this type of $\Phi $\ from Theorem \ref{Th8}
on.
\end{remark}

In Theorem \ref{Th3}, the moduli $\Phi \ $are used as in Remark \ref{Rem2}:

\begin{theorem}
\label{Th3} \cite[Theorem 2.3]{SY} Let $f:\left[ a,b\right] \rightarrow 
\mathbb{R}
$ be an uniformly convex function with modulus $\Phi :\left[ 0,b-a\right]
\rightarrow 
\mathbb{R}
_{+},$ $\left\{ x_{k}\right\} _{k=1}^{n}\subseteq \left[ a,b\right] $ be a
sequence and let $\pi $ be a permutation on $\left\{ 1,...,n\right\} $ such
that $x_{\pi \left( 1\right) }\leq x_{\pi \left( 2\right) }\leq ...\leq
x_{\pi \left( n\right) }$. Then the inequality%
\begin{equation}
\sum_{k=1}^{n}p_{k}f\left( x_{k}\right) -f\left(
\sum_{k=1}^{n}p_{k}x_{k}\right) \geq \sum_{k=1}^{n-1}p_{\pi \left( k\right)
}p_{\pi \left( k+1\right) }\Phi \left( x_{\pi \left( k+1\right) }-x_{\pi
\left( k\right) }\right) ,  \label{1.3}
\end{equation}%
holds for every convex combination $\sum_{k=0}^{n}p_{k}x_{k}$ of points $%
x_{k}\in \left[ a,b\right] $.
\end{theorem}

For $f=\Phi =x^{2},$ $x\geq 0$, we get for $n=2$,\ that there is an equality
in (\ref{1.3}) and for $n\geq 3,$\ $f=\Phi =x^{2},$ $x\geq 0$ (\ref{1.3}) is
mostly a strict inequality

\section{\textbf{Improved Jensen functional through Theorem \protect\ref{Th2}%
}}

In this section we assume that $x_{i}\leq x_{i+1}$, $i=1,...,n-1$.\ Theorems %
\ref{Th4}, \ref{Th5} and \ref{Th6} are easily proved by Theorem \ref{Th2}.
Theorems \ref{Th4}, and \ref{Th6} refine results in \cite{SA}.

\begin{theorem}
\label{Th4} Let $f$ be a convex function on $I$ and $\mathbf{x\in }\left[ a,b%
\right] ^{n}\subset I^{n}$. Let $\mathbf{p}=\left( p_{1},...,p_{n}\right) ,$
where $\ 0\leq \sum_{j=1}^{i}p_{j}\leq 1,$\ \ $i=1,...,n,$\ \ $%
\sum_{i=1}^{n}p_{i}=1,$then,%
\begin{equation}
0\leq \sum_{i=1}^{n}p_{i}f\left( x_{i}\right) -f\left(
\sum_{i=1}^{n}p_{i}x_{i}\right) \leq M^{\ast }\left( \frac{f\left( a\right)
+f\left( b\right) }{2}-f\left( \frac{a+b}{2}\right) \right) .  \label{2.1}
\end{equation}%
If in addition $p_{i}\geq 0$, $i=1,...,n$ then\ 
\begin{eqnarray}
0 &\leq &\sum_{i=1}^{n}p_{i}f\left( x_{i}\right) -f\left(
\sum_{i=1}^{n}p_{i}x_{i}\right) \leq M^{\ast }\left( \frac{f\left( a\right)
+f\left( b\right) }{2}-f\left( \frac{a+b}{2}\right) \right)  \label{2.2} \\
&<&2\left( \frac{f\left( a\right) +f\left( b\right) }{2}-f\left( \frac{a+b}{2%
}\right) \right) ,  \notag
\end{eqnarray}%
where $M^{\ast }$ is as defined in Theorem \ref{Th2} when $q_{0}=q_{n+1}=%
\frac{1}{2}$, $p_{0}=p_{n+1}=0$, $q_{i}=0$,$\ i=1,...,n$.
\end{theorem}

\begin{proof}
As we emphasized in the beginning of this section, we assume without loss of
generality that $x_{i}\leq x_{i+1}$, $i=1,...,n-1$.\ Choosing $%
p_{0}=p_{n+1}=0$ and $q_{0}=\frac{1}{2}=q_{n+1}$, $q_{i}=0,$ $i=1,...,n$,
and\ also $x_{0}=a$, $\ x_{n+1}=b$, then according to Theorem \ref{Th2} for $%
n+2$ terms we get that (\ref{2.1}) holds.

If in addition $p_{i}\geq 0$, $i=1,...,n$, then $M^{\ast }<2$ according to
Remark \ref{Rem1}, therefore (\ref{2.2}) holds.
\end{proof}

Theorem \ref{Th4} refines \cite[Theorem 2.2]{SA} where it is proved that $%
0\leq \sum_{i=1}^{n}p_{i}f\left( x_{i}\right) -f\left(
\sum_{i=1}^{n}p_{i}x_{i}\right) \leq 2\left( \frac{f\left( a\right) +f\left(
b\right) }{2}-f\left( \frac{a+b}{2}\right) \right) $.

\bigskip

In Section 3, we refine (\ref{2.3}) in Theorem \ref{Th5} for $\Phi $%
-uniformly convex functions $f$.

\begin{theorem}
\label{Th5} \cite[Theorem 2.3]{SA} Let $f$ be a convex function on $I=\left[
a,b\right] $ and $p,q>0,$ $p+q=1$. Then, 
\begin{eqnarray}
&&\min \left\{ p,q\right\} \left[ f\left( a\right) +f\left( b\right)
-2f\left( \frac{a+b}{2}\right) \right]  \label{2.3} \\
&\leq &pf\left( a\right) +qf\left( b\right) -f\left( pa+qb\right)  \notag \\
&\leq &\max \left\{ p,q\right\} \left[ f\left( a\right) +f\left( b\right)
-2f\left( \frac{a+b}{2}\right) \right]  \notag
\end{eqnarray}
\end{theorem}

\begin{proof}
Inequality (\ref{2.3}) is a particular case of Theorem \ref{Th1}.
\end{proof}

\begin{theorem}
\label{Th6} Let $f$ be a convex function on $I$ and $\mathbf{x\in }\left[ a,b%
\right] ^{n}\subset I^{n}$. Let $\mathbf{p}=\left( p_{1},...,p_{n}\right) ,$
where $\ 0\leq \sum_{j=1}^{i}p_{j}\leq 1,$\ \ $i=1,...,n,$\ \ $%
\sum_{i=1}^{n}p_{i}=1$, and let $m^{\ast }$\ and $M^{\ast }$\ as in Theorem %
\ref{Th2} when $q_{i}=\frac{1}{n},$ $i=1,...,n$.\ Then, 
\begin{eqnarray}
&&m^{\ast }\left[ \sum_{i=1}^{n}\frac{f\left( x_{i}\right) }{n}-f\left( 
\frac{\sum_{i=1}^{n}x_{i}}{n}\right) \right]  \label{2.4} \\
&\leq &J_{n}\left( f,\mathbf{x,p}\right)  \notag \\
&\leq &M^{\ast }\left[ \sum_{i=1}^{n}\frac{f\left( x_{i}\right) }{n}-f\left( 
\frac{\sum_{i=1}^{n}x_{i}}{n}\right) \right] .  \notag
\end{eqnarray}%
If $p_{i}\geq 0$, $i=1,...,n$ then 
\begin{eqnarray}
&&\underset{1\leq i\leq n}{\min }\left\{ p_{i}\right\} \left[
\sum_{i=1}^{n}f\left( x_{i}\right) -nf\left( \frac{\sum_{i=1}^{n}x_{i}}{n}%
\right) \right]  \label{2.5} \\
&\leq &m^{\ast }\left[ \sum_{i=1}^{n}\frac{f\left( x_{i}\right) }{n}-f\left( 
\frac{\sum_{i=1}^{n}x_{i}}{n}\right) \right]  \notag \\
&\leq &J_{n}\left( f,\mathbf{x,p}\right)  \notag \\
&\leq &M^{\ast }\left[ \sum_{i=1}^{n}\frac{f\left( x_{i}\right) }{n}-f\left( 
\frac{\sum_{i=1}^{n}x_{i}}{n}\right) \right]  \notag \\
&\leq &\underset{1\leq i\leq n}{\max }\left\{ p_{i}\right\} \left[
\sum_{i=1}^{n}f\left( x_{i}\right) -nf\left( \frac{\sum_{i=1}^{n}x_{i}}{n}%
\right) \right] .  \notag
\end{eqnarray}
\end{theorem}

\begin{proof}
As emphasized in the beginning of this section, we may assume, without loss
of generality, that $x_{i}\leq x_{i+1}$, $i=1,...,n-1.$ Using Theorem \ref%
{Th2} when $q_{i}=\frac{1}{n}$, $i=1,...,n$ we get that (\ref{2.4}) holds.

If $p_{i}\geq 0,$\ $i=1,...,n$\ then according to Remark \ref{Rem1} $M^{\ast
}<\underset{1\leq i\leq n}{\max }\left( \frac{p_{i}}{\left( \frac{1}{n}%
\right) }\right) $, \ unless $M^{\ast }=\underset{1\leq i\leq n}{\max }%
\left( \frac{p_{i}}{\left( \frac{1}{n}\right) }\right) =\frac{p_{1}}{\left( 
\frac{1}{n}\right) }$, or $M^{\ast }=\underset{1\leq i\leq n}{\max }\left( 
\frac{p_{i}}{\left( \frac{1}{n}\right) }\right) =\frac{p_{n}}{\left( \frac{1%
}{n}\right) }$ and $m^{\ast }>\underset{1\leq i\leq n}{\min }\left( \frac{%
p_{i}}{\left( \frac{1}{n}\right) }\right) $ unless $m^{\ast }=\underset{%
1\leq i\leq n}{\min }\left( \frac{p_{i}}{\left( \frac{1}{n}\right) }\right) =%
\frac{p_{1}}{\left( \frac{1}{n}\right) }$, or $m^{\ast }=\underset{1\leq
i\leq n}{\min }\left( \frac{p_{i}}{\left( \frac{1}{n}\right) }\right) =\frac{%
p_{n}}{\left( \frac{1}{n}\right) }$ . Hence (\ref{2.5}) holds.
\end{proof}

Theorem \ref{Th6} refines \cite[Theorem 2.4]{SA} where it is proved that
when $p_{i}\geq 0$, $i=1,...,n$, 
\begin{eqnarray*}
&&\underset{1\leq i\leq n}{\min }\left\{ p_{i}\right\} \left[
\sum_{i=1}^{n}f\left( x_{i}\right) -nf\left( \frac{\sum_{i=1}^{n}x_{i}}{n}%
\right) \right] \\
&\leq &J_{n}\left( f,\mathbf{x,p}\right) \leq \underset{1\leq i\leq n}{\max }%
\left\{ p_{i}\right\} \left[ \sum_{i=1}^{n}f\left( x_{i}\right) -nf\left( 
\frac{\sum_{i=1}^{n}x_{i}}{n}\right) \right]
\end{eqnarray*}%
holds.

\section{\textbf{\ Improved Jensen functional through }$\Phi $\textbf{%
-uniform convexity}}

From here on we deal with results related to $\Phi $-uniformly convex
functions which refine Theorem \ref{Th1} and Theorem \ref{Th5}.

We start this section discussing two inequalities, the first is proved by
using the basic properties of uniformly convex functions as proved in \cite%
{Z}, \cite[Theorem 1]{N} and \cite{A1}. The other inequality about uniformly
convexity is derived from results in \cite{SY}.\ 

\begin{remark}
\label{Rem3} In \cite[Theorem 2.1]{Z}, \cite[Theorem 1, Inequality (23)]{N}
and \cite[Inequality 8]{A1}, it is proved that the set of $\Phi $-uniformly
convex functions $f$\ which are continuously differentiable satisfy the
inequality 
\begin{equation}
f\left( y\right) -f\left( x\right) \geq f^{^{\prime }}\left( x\right) \left(
y-x\right) +\Phi \left( \left\vert y-x\right\vert \right)  \label{3.1}
\end{equation}%
and therefore, from (\ref{3.1}) also the inequality%
\begin{equation}
\sum_{i=1}^{n}p_{i}f\left( x_{i}\right) -f\left(
\sum_{i=1}^{n}p_{i}x_{i}\right) \geq \sum_{i=1}^{n}p_{i}\Phi \left(
\left\vert x_{i}-\sum_{j=1}^{n}p_{j}x_{j}\right\vert \right) .  \label{3.2}
\end{equation}
\end{remark}

From (\ref{3.2}) in Remark \ref{Rem3} we get refinements of Theorem \ref{Th1}
and Theorem \ref{Th5}.

\begin{theorem}
\label{Th7} Under the same conditions and definitions on\ $\ \mathbf{p},$ $%
\mathbf{q},$ $\mathbf{x},$ $m$ and $M$ as in Theorem \ref{Th1}, if $f:\left[
0,b\right) \rightarrow 
\mathbb{R}
,$ $0<b\leq \infty ,$ is contineously differential and is a $\Phi $%
-Uniformly Convex function, $\sum_{j=1}^{n}p_{j}x_{j}=\overline{x}_{p}$\ and 
$\sum_{j=1}^{n}q_{j}x_{j}=\overline{x}_{q}$, $\mathbf{x}\in \left[
0,b\right) ^{n},$\ then the following inequlities hold: 
\begin{equation}
J_{n}\left( f,\mathbf{x},\mathbf{p}\right) -mJ_{n}\left( f,\mathbf{x},%
\mathbf{q}\right) \geq m\Phi \left( \left\vert \overline{x}_{q}-\overline{x}%
_{p}\right\vert \right) +\sum_{i=1}^{n}\left( p_{i}-mq_{i}\right) \Phi
\left( \left\vert x_{i}-\overline{x}_{p}\right\vert \right) ,  \label{3.3}
\end{equation}%
and%
\begin{equation}
MJ_{n}\left( f,\mathbf{x},\mathbf{q}\right) -J_{n}\left( f,\mathbf{x},%
\mathbf{p}\right) \geq \sum_{i=1}^{n}\left( Mq_{i}-p_{i}\right) \Phi \left(
\left\vert x_{i}-\overline{x}_{q}\right\vert \right) +\Phi \left( \left\vert 
\overline{x}_{q}-\overline{x}_{p}\right\vert \right)  \label{3.4}
\end{equation}%
In the special case that $n=2$ and $m=\frac{p_{1}}{q_{1}}$ we get the
inequality 
\begin{eqnarray}
&&J_{2}\left( f,\mathbf{x},\mathbf{p}\right) -\frac{p_{1}}{q_{1}}J_{2}\left(
f,\mathbf{x},\mathbf{q}\right)  \label{3.5} \\
&\geq &\frac{p_{1}}{q_{1}}\Phi \left( \left\vert \overline{x}_{q}-\overline{x%
}_{p}\right\vert \right) +\left( p_{2}-\frac{p_{1}}{q_{1}}q_{2}\right) \Phi
\left( p_{1}\left\vert x_{2}-x_{1}\right\vert \right)  \notag \\
&=&m\Phi \left( \left\vert \overline{x}_{q}-\overline{x}_{p}\right\vert
\right) +\left( 1-m\right) \Phi \left( p_{1}\left\vert
x_{2}-x_{1}\right\vert \right) ,  \notag
\end{eqnarray}%
and 
\begin{eqnarray}
&&\frac{p_{2}}{q_{2}}J_{2}\left( f,\mathbf{x},\mathbf{q}\right) -J_{2}\left(
f,\mathbf{x},\mathbf{p}\right)  \label{3.6} \\
&\geq &\left( \frac{p_{2}}{q_{2}}q_{1}-p_{1}\right) \Phi \left( \left\vert
q_{2}\left( x_{2}-x_{1}\right) \right\vert \right) +\Phi \left( \left\vert 
\overline{x}_{q}-\overline{x}_{p}\right\vert \right) .  \notag
\end{eqnarray}%
When $q_{1}=q_{2}=\frac{1}{2}$ we get from (\ref{3.5}) and (\ref{3.6}) the
inequalities that refine Theorem \ref{Th5}:%
\begin{eqnarray}
&&J_{2}\left( f,\mathbf{x},\mathbf{p}\right) -2p_{1}\left( \frac{f\left(
x_{1}\right) +f\left( x_{2}\right) }{2}-f\left( \frac{x_{1}+x_{2}}{2}\right)
\right)  \label{3.7} \\
&\geq &2p_{1}\Phi \left( \left\vert \frac{x_{1}+x_{2}}{2}-\overline{x}%
_{p}\right\vert \right) +\left( 1-2p_{1}\right) \Phi \left( p_{1}\left\vert
x_{2}-x_{1}\right\vert \right)  \notag
\end{eqnarray}%
and 
\begin{eqnarray}
&&\frac{p_{2}}{q_{2}}\left( \frac{f\left( x_{1}\right) +f\left( x_{2}\right) 
}{2}-f\left( \frac{x_{1}+x_{2}}{2}\right) \right) -J_{2}\left( f,\mathbf{x},%
\mathbf{p}\right)  \label{3.8} \\
&\geq &\left( p_{2}-p_{1}\right) \Phi \left( \left\vert \left( \frac{%
x_{2}-x_{1}}{2}\right) \right\vert \right) +\Phi \left( \left\vert \frac{%
x_{1}+x_{2}}{2}-\overline{x}_{p}\right\vert \right) .  \notag
\end{eqnarray}

In the special case that $f\left( x\right) =\Phi \left( x\right) =x^{2},$ $%
x\geq 0$ we get equalities in (\ref{3.3}), (\ref{3.4}), (\ref{3.5}), (\ref%
{3.6}), (\ref{3.7}) and (\ref{3.8}).
\end{theorem}

\begin{proof}
To prove (\ref{3.3}) we define $\mathbf{y}$ as 
\begin{equation*}
y_{i}=\left\{ 
\begin{array}{ll}
x_{i}, & i=1,...,n \\ 
\sum_{j=1}^{n}q_{j}x_{j}, & i=n+1%
\end{array}%
\right. ,
\end{equation*}%
and $\mathbf{d}$ \ as 
\begin{equation*}
d_{i}=\left\{ 
\begin{array}{ll}
p_{i}-mq_{i}, & i=1,...,n \\ 
m, & i=n+1%
\end{array}%
\right. .
\end{equation*}%
\newline

Then (\ref{3.2}) for $\mathbf{y}$ and $\mathbf{d}$ is 
\begin{eqnarray*}
&&\sum_{i=1}^{n}\left( p_{i}-mq_{i}\right) f\left( x_{i}\right) +mf\left(
\sum_{i=1}^{n}q_{i}x_{i}\right) -f\left( \sum_{i=1}^{n}p_{i}x_{i}\right) \\
&=&\sum_{i=1}^{n+1}d_{i}f\left( y_{i}\right) -f\left(
\sum_{i=1}^{n+1}d_{i}y_{i}\right) \geq \sum_{i=1}^{n+1}d_{i}\Phi \left(
\left\vert y_{i}-\sum_{j=1}^{n+1}d_{j}y_{j}\right\vert \right)
\end{eqnarray*}%
\begin{equation*}
=\sum_{i=1}^{n}\left( p_{i}-mq_{i}\right) \Phi \left( \left\vert
x_{i}-\sum_{j=1}^{n}p_{j}x_{j}\right\vert \right) +m\Phi \left( \left\vert
\sum_{i=1}^{n}\left( p_{i}-q_{i}\right) x_{i}\right\vert \right)
\end{equation*}%
which is (\ref{3.3}) and therefore also (\ref{3.5}) and (\ref{3.7}) are
satisfied.

To get (\ref{3.4}), we choose $\mathbf{z}$ and $\mathbf{r}$ as 
\begin{equation*}
z_{i}=\left\{ 
\begin{array}{ll}
x_{i}, & i=1,...,n \\ 
\sum_{j=1}^{n}p_{j}x_{j}, & i=n+1%
\end{array}%
\right. ,
\end{equation*}%
and 
\begin{equation*}
r_{i}=\left\{ 
\begin{array}{ll}
q_{i}-\frac{p_{i}}{M}, & i=1,...,n \\ 
\frac{1}{M}, & i=n+1%
\end{array}%
\right. .
\end{equation*}

Then, as $f$ is $\Phi $-uniformly convex, $\sum_{i=1}^{n+1}r_{i}=1,$ \ $%
r_{i}\geq 0,$ and $\sum_{i=1}^{n}q_{i}x_{i}=\sum_{i=1}^{n+1}r_{i}z_{i}$,\ we
get that 
\begin{eqnarray*}
&&\sum_{i=1}^{n}\left( q_{i}-\frac{p_{i}}{M}\right) f\left( x_{i}\right) +%
\frac{1}{M}f\left( \sum_{i=1}^{n}p_{i}x_{i}\right) -f\left(
\sum_{i=1}^{n}q_{i}x_{i}\right) \\
&=&\sum_{i=1}^{n+1}r_{i}f\left( z_{i}\right) -f\left(
\sum_{i=1}^{n+1}r_{i}z_{i}\right) \\
&\geq &\sum_{i=1}^{n+1}r_{i}\Phi \left( \left\vert
z_{i}-\sum_{i=1}^{n+1}r_{i}z_{i}\right\vert \right) \\
&=&\sum_{i=1}^{n}\left( q_{i}-\frac{p_{i}}{M}\right) \Phi \left( \left\vert
x_{i}-\sum_{j=1}^{n}q_{j}x_{j}\right\vert \right) +\frac{1}{M}\Phi \left(
\left\vert \sum_{i=1}^{n}\left( p_{i}-q_{i}\right) x_{i}\right\vert \right)
\end{eqnarray*}%
which is equivalent to (\ref{3.4}), and therefore also (\ref{3.6}) and (\ref%
{3.8}) are satisfied. \ The proof of the theorem is complete.
\end{proof}

The proof of Theorem \ref{Th7} is similar to the proof of \cite[Theorem 3]%
{AD}, there it is about superquadratic functions.

In Theorem \ref{Th7} we used (\ref{3.2}) to prove a refinement of Theorem %
\ref{Th1} through $\Phi $-uniformly convex functions.

From Theorem \ref{Th3} we get another refinement of Theorem \ref{Th1} and
Theorem \ref{Th5} proved in the following theorem where $\Phi $ satisfies
the conditions of Remark \ref{Rem2}:

\begin{theorem}
\label{Th8} Under the same conditions and definitions on\ $\ \mathbf{p},$ $%
\mathbf{q},$ $\mathbf{x},$ $m$ as in Theorem \ref{Th1}, if $f:\left[
a,b\right) \rightarrow 
\mathbb{R}
,$ is a $\Phi $-Uniformly Convex function, where $\Phi :\left[ 0,b-a\right)
\rightarrow 
\mathbb{R}
_{+},$ $\sum_{j=1}^{n}p_{j}x_{j}=\overline{x}_{p}$\ and $%
\sum_{j=1}^{n}q_{j}x_{j}=\overline{x}_{q}$, $\mathbf{x}\in \left[ 0,b\right)
^{n}$ and\ also $x_{1}\leq x_{2}\leq ,...,\leq x_{k-1}\leq
\sum_{i=1}^{n}q_{i}x_{i}\leq x_{k}\leq x_{k+1}\leq ,...,\leq x_{n}$,\ then
the following inequlity holds:%
\begin{eqnarray}
&&  \label{3.9} \\
&&J_{n}\left( f,\mathbf{x},\mathbf{p}\right) -mJ_{n}\left( f,\mathbf{x},%
\mathbf{q}\right)  \notag \\
&\geq &\sum_{i=1,i\neq k-1}^{n}\left( p_{i}-mq_{i}\right) \left(
p_{i+1}-mq_{i+1}\right) \Phi \left( x_{i+1}-x_{i}\right)  \notag \\
&&+m\left( p_{k-1}-mq_{k-1}\right) \Phi \left(
\sum_{i=1}^{n}q_{i}x_{i}-x_{k-1}\right) +m\left( p_{k}-mq_{k}\right) \Phi
\left( x_{k}-\sum_{i=1}^{n}q_{i}x_{i}\right) .  \notag
\end{eqnarray}
\end{theorem}

\begin{proof}
It is given here that the sequence $\mathbf{x}\ $is increasing. Therefore we
can apply the inequality (\ref{1.3}) in Theorem \ref{Th3} for the also
increasing $(n+1)$-tuple $\mathbf{y=}\left( y_{1},...,y_{n+1}\right) $ \ 
\begin{equation}
y_{i}=\left\{ 
\begin{array}{ll}
x_{i,} & i=1,...,k-1 \\ 
\sum_{j=1}^{n}q_{j}x_{j},\quad & i=k \\ 
x_{i-1}, & i=k+1,...,n+1%
\end{array}%
\right.  \label{3.10}
\end{equation}%
and to 
\begin{equation}
d_{i}=\left\{ 
\begin{array}{ll}
p_{i}-mq_{i}, & i=1,...,k-1 \\ 
m, & i=k \\ 
p_{i-1}-mq_{i-1},\quad & i=k+1,...,n+1%
\end{array}%
\right. ,  \label{3.11}
\end{equation}%
where $m=\underset{1\leq i\leq n}{\min }\left( \frac{p_{i}}{q_{i}}\right) $,
and we get from Theorem \ref{Th3} that 
\begin{eqnarray*}
&&J_{n}\left( f,\mathbf{x},\mathbf{p}\right) -mJ_{n}\left( f,\mathbf{x},%
\mathbf{q}\right) \\
&=&\sum_{i=1}^{n+1}d_{i}f\left( y_{i}\right) -f\left(
\sum_{i=1}^{n+1}d_{i}y_{i}\right) \geq \sum_{i=1}^{n+1}d_{i}d_{i+1}\Phi
\left( \left\vert y_{i+1}-y_{i}\right\vert \right) .
\end{eqnarray*}%
This inequality means by (\ref{3.10}) and (\ref{3.11}) that%
\begin{eqnarray*}
&&\sum_{i=1}^{k-1}\left( p_{i}-mq_{i}\right) f\left( x_{i}\right) +mf\left(
\sum_{i=1}^{n}p_{i}x_{i}\right) +\sum_{i=k}^{n}\left( p_{i}-mq_{i}\right)
f\left( x_{i}\right) \\
&\geq &\sum_{i=1,i\neq k-1}^{n}\left( p_{i}-mq_{i}\right) \left(
p_{i+1}-mq_{i+1}\right) \Phi \left( x_{i+1}-x_{i}\right) \\
&&+m\left( p_{k-1}-mq_{k-1}\right) \Phi \left(
\sum_{i=1}^{n}q_{i}x_{i}-x_{k-1}\right) +m\left( p_{k}-mq_{k}\right) \Phi
\left( x_{k}-\sum_{i=1}^{n}q_{i}x_{i}\right) .
\end{eqnarray*}%
Hence (\ref{3.9}) holds and the proof of the theorem is complete.
\end{proof}

In Theorem \ref{Th9} we get a refinement of the left handside of the
inequality (\ref{2.3}) in Theorem \ref{Th5} for $\Phi $-uniformly convex
functions:

\begin{theorem}
\label{Th9} Let $f$ be $\Phi $-uniformly convex. Then the inequality 
\begin{eqnarray}
&&p_{1}f\left( a\right) +p_{2}f\left( b\right) -f\left( p_{1}a+p_{2}b\right)
-m\left( q_{1}f\left( a\right) +q_{2}f\left( b\right) -f\left(
q_{1}a+q_{2}b\right) \right)  \label{3.12} \\
&\geq &m\left( p_{2}-mq_{2}\right) \Phi \left( q_{1}\left( b-a\right)
\right) =m\left( 1-m\right) \Phi \left( q_{1}\left( b-a\right) \right) 
\notag
\end{eqnarray}%
holds when $\frac{p_{1}}{q_{1}}=m\leq \frac{p_{2}}{q_{2}}$, $p_{i}\geq 0$,$\
q_{i}>0$, $i=1,2$, $p_{1}+p_{2}=q_{1}+q_{2}=1$.

In the special case for $q_{1}=q_{2}=\frac{1}{2}$ we get the inequality%
\begin{eqnarray}
&&p_{1}f\left( a\right) +p_{2}f\left( b\right) -f\left( p_{1}a+p_{2}b\right)
-2p_{1}\left[ \frac{f\left( a\right) +f\left( b\right) }{2}-f\left( \frac{a+b%
}{2}\right) \right]  \label{3.13} \\
&\geq &2p_{1}\left( 1-2p_{1}\right) \Phi \left( \frac{b-a}{2}\right) , 
\notag
\end{eqnarray}%
and the maximum of the right handside of (\ref{3.13}) is obtained when $%
p_{1}=\frac{1}{4}$ and (\ref{3.13}) is 
\begin{eqnarray}
&&\frac{1}{4}f\left( a\right) +\frac{3}{4}f\left( b\right) -f\left( \frac{1}{%
4}a+\frac{3}{4}b\right)  \label{3.14} \\
&&-\frac{1}{2}\left( \frac{f\left( a\right) +f\left( b\right) }{2}-f\left( 
\frac{a+b}{2}\right) \right)  \notag \\
&\geq &\frac{1}{4}\Phi \left( \frac{x_{2}-x_{1}}{2}\right) .  \notag
\end{eqnarray}
\end{theorem}

\begin{proof}
The left handside of (\ref{3.12}) can be rewritten when $\frac{p_{1}}{q_{1}}%
=m$,$\ x_{1}=a$, $x_{2}=b$\ as 
\begin{equation*}
\left( p_{2}-mq_{2}\right) f\left( x_{2}\right) +mf\left(
q_{1}x_{1}+q_{2}x_{2}\right) -f\left( p_{1}x_{1}+p_{2}x_{2}\right)
\end{equation*}

Denoting%
\begin{equation}
\begin{tabular}{ll}
$d_{1}=m,$ & $y_{1}=q_{1}x_{1}+q_{2}x_{2}$ \\ 
$d_{2}=p_{2}-mq_{2},$ & $y_{2}=x_{2},$%
\end{tabular}
\label{3.15}
\end{equation}%
we take into consideration that $\frac{p_{1}}{q_{1}}=m\leq \frac{p_{2}}{q_{2}%
}$, and because $\sum_{i=1}^{2}p_{i}x_{i}=\sum_{i=1}^{2}d_{i}y_{i}$ we get
from (\ref{3.15}) and Definition \ref{Def1} that the following\ inequality
holds: 
\begin{eqnarray}
&&p_{1}f\left( x_{1}\right) +p_{2}f\left( x_{2}\right) -f\left(
p_{1}x_{1}+p_{2}x_{2}\right)  \label{3.16} \\
&&-m\left( q_{1}f\left( x_{1}\right) +q_{2}f\left( x_{2}\right) -f\left(
q_{1}x_{1}+q_{2}x_{2}\right) \right)  \notag \\
&=&\sum_{i=1}^{2}d_{i}f\left( y_{i}\right) -f\left(
\sum_{i=1}^{2}d_{i}y_{i}\right)  \notag \\
&=&mf\left( q_{1}x_{1}+q_{2}x_{2}\right) +\left( p_{2}-mq_{2}\right) f\left(
x_{2}\right)  \notag \\
&\geq &m\left( p_{2}-mq_{2}\right) \Phi \left( \left( x_{2}-x_{1}\right)
q_{1}\right)  \notag \\
&=&\frac{p_{1}}{q_{1}}\left( 1-p_{1}-\frac{p_{1}}{q_{1}}q_{2}\right) \Phi
\left( q_{1}\left( x_{2}-x_{1}\right) \right)  \notag \\
&=&m\left( 1-m\right) \Phi \left( q_{1}\left( x_{2}-x_{1}\right) \right) . 
\notag
\end{eqnarray}%
Hence, (\ref{3.12}) is proved and therefore also (\ref{3.13}). The maximum
of the right handside of (\ref{3.13}) is obtained when $p_{1}=\frac{1}{4}$
therefore the right handside of (\ref{3.14}) gives the best refinement of
the left handside of (\ref{2.3}) in Theorem \ref{Th5} for $\Phi $-uniformly
convex functions.
\end{proof}

\begin{remark}
\label{Rem4} Inequality (\ref{3.7}) is a refinement of the left handside of (%
\ref{2.3}). However when $\frac{\Phi \left( x\right) }{x^{2}}$ is
increasing, it is a weaker refinement than the refinement (\ref{3.14}) in
Theorem \ref{Th9} obtained from Definition \ref{Def1}.

When we take $f\left( x\right) \ $to be $f\left( x\right) =\Phi \left(
x\right) =x^{2}$, $x\geq 0$ we get an equality all over Theorem \ref{Th7}
for all $n=2,3,...$. But in Theorem \ref{Th8} we get always equality only
for $f\left( x\right) =\Phi \left( x\right) =x^{2}$, for $n=2$.\ 
\end{remark}

\bigskip

\bigskip

\end{document}